%% file: arXiv_ClosureModel_DVMS.tex
\documentclass[11pt]{article}
\usepackage[small,compact]{titlesec}
\usepackage{amsmath,amssymb,amsfonts,mathabx,setspace}
\usepackage{graphicx,epsfig,wrapfig,caption,epsfig,subcaption,sidecap}
\usepackage{url,color,verbatim,algorithmic}
\usepackage[ruled]{algorithm}  
\usepackage[sort,compress]{cite}
\usepackage[scaled]{helvet}
\usepackage[T1]{fontenc}
\usepackage[ansinew]{inputenc}
\usepackage[bookmarks=true, bookmarksnumbered=true, colorlinks=true,   pdfstartview=FitV,
linkcolor=blue, citecolor=blue, urlcolor=blue]{hyperref}
\usepackage[english]{babel}  
\usepackage{enumitem}
\usepackage{multirow}
\let\oldbibliography\thebibliography
\renewcommand{\thebibliography}[1]{\oldbibliography{#1}
\setlength{\itemsep}{0pt}}

\usepackage{perpage} \MakePerPage{footnote}

\usepackage{full page}

\def\mE{\mathcal{E}}
\def\wA{ \widetilde{A}}  
\def\wB{\widetilde{B}}   
\def\hatA{ \widehat{A}}  
\def\hatB{\widehat{B}}   
\def\bv{\boldsymbol{v}}
\def\bF{\boldsymbol{F}}
\def\bxi{\boldsymbol{\xi}}
\newcommand{\argmin}[1]{\underset{#1}{\operatorname{arg}\operatorname{min}}\;}

\def\R{\mathbb{R}}                                      \def\E{\mathbb{E}}

\newcommand{\innerp}[1]{\langle{#1}\rangle}

\newcommand{\floor}[1]{\lfloor{#1}\rfloor}

\newenvironment{proof}[1][Proof]{\noindent\textbf{#1.} }{\ \rule{0.5em}{0.5em}}

\numberwithin{equation}{section}

\usepackage{authblk} 

\usepackage{color}

\input{notation}

\title{Stochastic Data-Driven Variational Multiscale \\ Reduced Order Models}

\author[1]{Fei Lu}
\author[2]{Changhong Mou}
\author[3]{Honghu Liu}
\author[4]{Traian Iliescu}

\affil[1]{\footnotesize  Department of Mathematics, Johns Hopkins University, Baltimore, MD 21218, USA (\href{feilu@math.jhu.edu}{feilu@math.jhu.edu})}
\affil[2]{\footnotesize Department of Mathematics, University of Wisconsin-Madison, Madison, WI 53705, USA (\href{cmou3@wisc.edu}{cmou3@wisc.edu})}
\affil[3]{\footnotesize  Department of Mathematics, Virginia Tech, Blacksburg, VA 24061, USA (\href{hhliu@vt.edu}{hhliu@vt.edu})}
\affil[4]{\footnotesize  Department of Mathematics, Virginia Tech, Blacksburg, VA 24061, USA (\href{iliescu@vt.edu}{iliescu@vt.edu})}

 \date{\today}

\begin{document}
\maketitle

\begin{abstract}

 Trajectory-wise data-driven reduced order models (ROMs) tend to be sensitive to training data, and thus lack robustness. We propose to construct a robust stochastic ROM closure (S-ROM) from data consisting of multiple trajectories from random initial conditions. The S-ROM is a low-dimensional time series model for the coefficients of the dominating proper orthogonal decomposition (POD) modes inferred from data. Thus, it achieves reduction both space and time, leading to simulations orders of magnitude faster than the full order model. 
 We show that both the estimated POD modes and parameters in the S-ROM converge when the number of trajectories increases. Thus, the S-ROM is 
 robust when the training data size increases.  
 We demonstrate the S-ROM on a 1D Burgers equation with a viscosity $\nu= 0.002$ and with random initial conditions. The numerical results verify the convergence. Furthermore, the S-ROM makes accurate trajectory-wise predictions from new initial conditions and with a prediction time far beyond the training range, and it quantifies the spread of uncertainties due to the unresolved scales.
\end{abstract}

\tableofcontents

\section{Introduction} 

Data-driven methods, combined with structures in physical systems, provide effective tools for the construction of closure reduced order models (ROMs) \cite{mou2020dd-vms-rom,Lu20Burgers,MH13,MK02} that bring down computational cost by orders of magnitude from the full order models (FOM). The data are a single or multiple solutions  of a FOM 
\begin{equation}\label{eq:FOM-Data}
\dot \bu  = \bff(\bu),\quad \Rightarrow  \quad  \{\bu(t)^{(m)}, t\in [0,T]\}_{m=1}^M,
\end{equation}
where $\bu$ can be either a high-dimensional state vector for a system of stochastic/ordinary differential equations or a solution to a partial differential equation.  The closure ROMs are low-dimensional models that are constructed from data. They approximate the FOM solution $\bu$ by $\bu(t,x) \approx \sum_{i=1}^r a_i(t)\bphi_i(x)$, where both the basis functions  $\{\bphi_1,\ldots, \bphi_r\}$ and the ROM closure for the coefficient vector $\ba = (a_1,\ldots,a_r)$: 
\begin{equation}\label{eq:ROM-closure}
    \dot \ba = \bF(\ba) + \text{ Closure}(\ba),
\end{equation}
are estimated from data. Here the term $\bF$ often comes from a truncated representation of $\bff$ in the FOM, and the term $\text{Closure}(\ba)$ accounts for the effects of the unresolved modes $\{\bphi_{r+1},\ldots, \bphi_n,\ldots\}$.

The closure problem is ubiquitous in dimension-reduced modeling of high- or infinite-dimensional nonlinear dynamics. The task is to account for the unresolved scales crucial for the dynamics, such as the sub-grid scales in the classical numerical discretization of turbulent flows (e.g., finite element or finite volume methods). There are hundreds of closure models in classical CFD, e.g., large eddy simulation (LES) \cite{BG02,BIL05,DN07,berselli2011horizontal}. These LES closure models are built around the physical insight stemming from Kolmogorov's statistical theory of turbulence. Unfortunately, this physical insight is generally posed in the Fourier series setting, thus not directly applicable to ROMs whose basis functions are estimated from data. Therefore, these ROM closure models are deprived of many tools that represent the core of most LES closure models.

The combination of physical insights and data provides a promising direction for ROM closure modeling. It is poised to extend the reach of physical insights, reduce the size of training data, and produce robust predictive closure models. This is an extremely active research area and many important research have  demonstrated the power of this combination (see e.g., \cite{aubryDynamicsCoherent1988,sirisupSpectralViscosity2004,willcoxUnsteadyFlow2006,wangProperOrthogonal2012,san2014proper,iliescu2018regularized,xie2018data,giere2015supg,mohebujjaman2019physically,carlbergEfficientNonlinear2011,gunzburger2017ensemble,choi2019space,CLW15_vol2}  for a biased sample). In particular, deterministic ROM closure models have been constructed with proper orthogonal decomposition (POD) basis functions, Galerkin projections and least squares closures (see e.g., \cite{carlbergEfficientNonlinear2011,mou2020dd-vms-rom}). The resulting ROM closures accurately fit the data, but they tend to have limited generalizability beyond the training data set and do not represent the uncertainty due to the unresolved scales. 
Stochastic ROM closure models have been constructed from multi-trajectory data via statistical inference \cite{CL15,LLC17,Lu20Burgers,LinLu21}. While they quantify the spread of uncertainty, they are only for pre-selected large scale variables. 

We propose to construct data-driven stochastic ROM (S-ROM) closure for deterministic systems from multi-trajectory data with random initial conditions, so that the S-ROM can make accurate predictions beyond training datasets and quantify the spread of uncertainty. The construction consists of two steps. First, we identify the dominating modes (basis functions) from the data by POD. Then, we infer a stochastic closure model for the evolution of the coefficients of the dominating modes. The S-ROM approximates the discrete-time flow map of these modes, and its parameter is estimated by maximizing the likelihood of the data. It is a natural probabilistic extension of the variational approach from deterministic ROMs \cite{mou2020dd-vms-rom,mou2020data} to stochastic models, and an extension of the stochastic closure models  in \cite{CL15,LLC17,Lu20Burgers} from pre-selected basis to data-deduced basis.    

This study focuses on systems with quadratic nonlinearity and considers S-ROMs depending linearly on the parameters. As a result, we can efficiently construct the S-ROM by least squares (with regularization when the normal matrix is ill-conditioned). Furthermore, we prove that the estimators of the POD basis and the parameters converge at the rate $M^{-1/2}$, where $M$ denotes the number of independent trajectories. 

We demonstrate this new approach on a 1D Burgers equation with random initial conditions. Numerical results verify the convergence of the estimators of the POD basis and the parameters of the S-ROM. In contrast, the single-trajectory estimator varies largely between trajectories.   Furthermore, the S-ROM makes accurate predictions from new random initial conditions for a longer time than the training interval, and it quantifies the spread of the uncertainty. In addition, the S-ROM is adaptive to time-step size, and it can tolerate a relatively large time-step size because it accounts for the discretization error by utilizing the information from data.  

This study provides a parametric inference examination
of the success of the combination of physical insights and data. 
As a data-driven approach, this study 
makes a new contribution to 
the recent advances on utilizing statistical learning and machine learning tools to construct ROM closure models, ranging from Mori-Zwanzig formalism \cite{CHK02,chorin2003conditional,stinisRenormalizedMoriZwanzigreduced2015,zhu2018estimation},  non-Markovian modeling \cite{KCG15,CL15,leiDatadrivenParameterization2016,LinLu21}, filtering and model error quantification \cite{GHM10,MH13,Har16}, and conditional Gaussian models \cite{chen2018conditional,chen2022conditional}, to machine learning methods \cite{chen2022physics,parishParadigmDatadriven2016,ma2018model,san2019artificial,harlim2020machine,levine2021framework,feng2022auto}, to name just a few.  
Furthermore, it can be viewed as an approximation of the flow operator of the full order model, in the context of operator inference by machine learning methods \cite{klus2018data}.  

Our main contributions are:  \vspace{-2mm}
\begin{enumerate}[leftmargin=*]\setlength\itemsep{-1mm}
\item We examine a parametric inference approach, which combines physical insights with data, to construct stochastic ROM closures from multi-trajectory data. The parameters are efficiently estimated by regression with regularization. We prove and numerically verify the convergence of the POD modes and the parameter estimators as the number of data trajectories increases. 
\item  We show that the training trajectories can be short trajectories for autonomous systems, and the constructed S-ROM can make predictions far beyond the training time interval. Furthermore, the S-ROM can represent the spread of the uncertainty due to unresolved scales, making it suitable for data assimilation tasks. 
\item We demonstrate that the S-ROM is adaptive to both resolution and time step size. Thus, it leads to reduction in both space and time, achieving efficient simulations that can be orders of magnitudes faster than the FOM. 
\end{enumerate}

The rest of the paper is organized as follows:  
we introduce the method for the construction of the S-ROM in Section \ref{sect_framework}. Section \ref{sec:Burgers} investigates S-ROM in the context of a Burgers equation, systematically examining the convergence of estimators and the performance of the S-ROM in trajectory-wise prediction and ensemble prediction. Conclusions and final remarks are
presented in Section \ref{sec:conclusion}.

\section{Stochastic DD-VMS-ROM} \label{sect_framework}

In this section, we introduce the method for the construction of a stochastic data-driven variational multiscale reduced order model (DD-VMS-ROM). It leads to a robust stochastic ROM (S-ROM) that quantifies the uncertainty from the unresolved scales and the randomness from the initial condition. We construct the S-ROM from multiple-trajectory data in two stages: (1) construct the basis functions (orthogonal modes) from the data by proper orthogonal decomposition (POD); (2) infer the parameters in the S-ROM, which is a time series model describing the dynamics of the coefficients of the basis functions, by maximizing the likelihood of the multi-trajectory data.  
 Thus, the S-ROM builds on data-deduced basis functions, a parametric structure derived from the full model, and a parametric inference. It is a natural probabilistic extension of the variational approach from deterministic ROMs \cite{mou2020dd-vms-rom,mou2020data} to stochastic models, and an extension of the stochastic closure models  in \cite{CL15,LLC17,Lu20Burgers} from pre-selected spectral basis to data-deduced basis.

\subsection{Variational multiscale methods and flow map approximation} \label{sec:Previous}

\paragraph{Variational multiscale methods.} The variational multiscale (VMS) methods \cite{hughes1998variational,hughes2000large,hughes2001large}
 increase the accuracy of classical Galerkin approximation by introducing hierarchical spaces and coupling terms to model the effects of the unresolved scales. To illustrate the VMS methodology, suppose that the FOM in \eqref{eq:FOM-Data} is a nonlinear PDE and consider its variational (weak) form
\begin{equation}\label{eq:VMS_full}
\innerp{\dot \bu, \bv} = \innerp{\bff(\bu), \mathbf{v}}, \quad \bv \in \bX,
\end{equation}
where $\bX$ is an appropriate Hilbert space and $\innerp{\cdot,\cdot}$ denotes its inner product. The VMS projects the full model onto subspaces $\bX_1,\bX_2, \bX_3,\ldots$, which build a sequence of hierarchical spaces of increasing resolutions $\bX_1, \bX_1\oplus\bX_2,\bX_1\oplus\bX_2\oplus\bX_3,\ldots$. The projection leads to separate equations on each space, and the goal is to solve for the components of $\bu$ that live in the space $\bX_1$ (which represents the resolved scales), denoted by $\bu_1$, whose equation is a low-dimensional system
\begin{equation}\label{eq:u1FOM}
\innerp{\dot \bu_1, \bv_1}  = \innerp{\dot \bu, \bv_1} = \innerp{\bff(\bu), \bv_1} = \innerp{\bff(\bu_1), \bv_1} +  \innerp{\bff(\bu)-\bff(\bu_1), \bv_1} , \quad \forall \bv_1 \in \bX_1.
\end{equation}
The above system is not closed since, due to the nonlinearity of $\bff(\bu)$, the term $\innerp{\bff(\bu)-\bff(\bu_1), \bv_1}$ involves components that do not live in $\bX_1$. Thus, a VMS closure model of $\bu_1$ aims to approximate  $\innerp{\bff(\bu)-\bff(\bu_1), \bv_1}$ in terms of $\bu_1$.   

To construct the VMS closure, we identify a closed dynamical system to describe the dynamics of the coefficients of $\bu_1$ in an orthonormal basis of $\bX_1$. More precisely, suppose that $\bX$ has a complete orthonormal basis  $\{\bphi_1,\ldots,\bphi_r, \ldots\}$ and $\bX_1 = \mathrm{span} \{\bphi_i\}_{i=1}^r$, and write the solution as $\bu(t,x)= \sum_{i=1}^\infty a_i(t)\bphi_i(x) $  and $\bu_1(t,x)= \sum_{i=1}^r a_i(t)\bphi_i(x) $.  Denote $\ba = (a_1,a_2, \ldots,a_r)$ and $\bb= (a_{r+1},\ldots, a_{r+n},\ldots)$, where $\ba_i = \innerp{\dot \bu, \bphi_i} $ for $i\geq 1$.  Then, Eq.~\eqref{eq:u1FOM} is equivalent to the differential system 
\begin{equation}\label{eq:u1FOM_v2}
\dot \ba  =\bF(\ba) +  \bG(\ba,\bb),   
\end{equation}
which is a differential system on $\R^r$ with  
\begin{equation}
\bF(\ba)_i = \innerp{\bff(\bu_1), \bphi_i}, \qquad \bG(\ba,\bb)_i =\innerp{\bff(\bu), \bphi_i} - \innerp{\bff(\bu_1), \bphi_i},  \qquad i=1,\ldots,r. 
\end{equation}
Thus, the task of a closure problem is to approximate the term $\bG(\ba,\bb)$ by either a function of $\ba$ or a functional of the path of $\ba$. Similarly, a hierarchical multiscale closure model describes the dynamics of the coefficients of the $\bu$ in hierarchical spaces.

There are two main challenges: (1) the hierarchical spaces, i.e., their basis functions, can be difficult to construct in numerical methods with non-orthogonal basis functions (e.g., finite element or finite volume methods); and (2) numerical closure models, which account for the effects of the unresolved scales (i.e., the term $\bG(\ba,\bb)$), are difficult to construct from the first principles.

These challenges promote recent developments that combine data and physical insights in the construction of ROM closure \cite{Zwa01,CH13,CKK98,KCG15,GHM10,MH13}. In particular, the data-driven variational multiscale reduced order model (DD-VMS-ROM) in \cite{xie2018data,mou2020dd-vms-rom} constructs the basis functions from data by POD and identifies the closure method through regression, leading to an accurate deterministic model for each trajectory. Meanwhile, stochastic closure models for the coefficients of prescribed spectral-Galerkin basis have been constructed in \cite{LLC17,LinLu21,Lu20Burgers} from data consisting of multiple trajectories by statistical inference methods, leading to robust predictive models. In the following, we briefly review these two approaches, highlighting their connections that will foster a combination of their strengths to construct the stochastic DD-VMS-ROM in the next section. 

\paragraph{DD-VMS-ROM.}
The DD-VMS-ROM method introduced in \cite{xie2018data,mou2020dd-vms-rom} constructs a ROM closure that fits data consisting of a single-trajectory: it first constructs basis functions from data by POD, then it estimates the parameters in the closure model. 
For example, when $f(\bu)$ in \eqref{eq:VMS_full} consists of only linear and quadratic terms as typically encountered in fluid flow problems, the two-scale 2S-DD-VMS-ROM~\cite{xie2018data} is
\begin{align} \label{eq:2S-ROM}  
\dot \ba
= (A + \wA)\ba 
+ \ba^\top (B + \wB) \ba, 
\end{align}
where $A$ and $\wA$ are $r\times r$ matrices, while $B$ and $\wB$ are $r\times r \times r$ tensors. 
The term $A \ba + \ba^\top B \ba$ corresponds to $\bF(\ba)$ in \eqref{eq:u1FOM_v2}, which represents the Galerkin projection of $\bff(\bu)$ onto the subspace $\bX_1$. 
The additional term $\wA \ba +  \ba^\top \wB \ba$ aims to approximate the model error $\bG(\ba,\bb)$ in \eqref{eq:u1FOM_v2} by quadratic polynomials of the resolved scales, and the entries in $\wA$ and $\wB$ are estimated from data, consisting of time snapshots in a single solution. The above two-scale model has also been extended to three-scale 3S-DD-VMS-ROM~\cite{mou2020dd-vms-rom} to include finer interactions in hierarchical spaces. 

\paragraph{Inference-based closures.}
The stochastic model reduction method \cite{LLC17,LinLu21,Lu20Burgers} constructs ROM closures by inferring a time series model that approximates the \emph{discrete-time flow map} of the resolved scales. The discrete-time flow map of $\ba$ is obtained by integrating \eqref{eq:u1FOM_v2} on an time interval $[t_{n-1},t_n]$,  
\begin{align} \label{eq:flowmap}
\ba(t_n) & = \ba(t_{n-1}) + \int_{t_{n-1}}^{t_n} \left[ \bF(\ba(s)) +  \bG(\ba(s),\bb(s))\right] ds \notag
\\ & = \mathcal{F}(\ba(t_{n-1}), \bb(t_{n-1})) 
 \approx F_\theta(\ba(t_{n-p:n-1}), \bxi_{n-p:n-1}),
\end{align}
where the flow map $\mathcal{F}(\ba(t_{n-1}), \bb(t_{n-1})) $ depends on the unresolved scale $\bb(t_{n-1})$. It is approximated by a functional $F_\theta$ depending on the past of $\ba$, denoted by a vector with time lags $\ba(t_{n-p:n-1})=\left(\ba(t_{n-p}),\ldots, \ba(t_{n-1})\right)$,  since $\bb(t_{n-1})$ depends on the history of $\ba$. Here $F_\theta$ is a parametric function with parameter $\theta$ estimated from data, and $\bxi_n$ is assumed to be sequence of independent identically distributed (IID) Gaussian noise to quantify the uncertainty. In particular, utilizing the quadratic terms, a parametric form of $F_\theta$ is shown to be effective for a stochastic Burgers equation \cite{Lu20Burgers}:  
\begin{align}
F_\theta(\ba(t_{n-p:n-1}), \bxi_{n-p:n-1}) \approx c_1\ba(t_{n-1}) + c_R R^\delta(\ba) +c^w G(\ba)  + \bxi_{n-1},
\end{align}
 where the parameters $(c_1,c_R,c^w)$ are diagonal matrices, $R^{\delta} $ comes from the RK4 (the fourth-order Runge-Kutta) integrator of the $r$-mode truncated system, and 
 $G(\ba)$ denotes a vector with entries 
 \begin{equation}  \label{eq:ks-ansatz-d}
G(\ba)_k =  \sum_{l \in \mathcal{A}_{k,r}} {\widetilde \ba^{n-1}_l \widetilde \ba^{n-j}_{k-l}} 
 \,  \text{ with } \widetilde{\ba}^{n-j}_{k}=
    \left\{
    \begin{array}{ll}
      \ba^{n-j}_k~, & 1\leq k\leq r; \\[1ex]
     \sum_{|l|\leq r, |k-l|\leq r}\ba^{n-j}_{k-l} \ba^{n-j}_{l} , & r < k \leq 2r,
    \end{array}
    \right.
  \end{equation}
  where $ \mathcal{A}_{k,r} = \{l: |k-l|\leq r, r< |l| \leq 2r  \text{ or }  |l|\leq r, r< |k-l| \leq 2r \}$ and $ \ba^{m}_l =\ba_l(t_{m})$ for each $l,m$. 
 The time series model, in the form of a nonlinear autoregression moving average model (NARMA), is inferred from multiple-trajectory data. It takes into account both the model error and the numerical discretization error. Thus, it can reproduce the statistics and the dynamics of the resolved scales. In particular, it can be used for ensemble prediction with uncertainty quantified \cite{Lu20Burgers,chen2022conditional}. 

Note that when $\bff(\bu)$ consists of quadratic nonlinearities, the structure of the DD-VMS-ROM in~\cite{mou2020dd-vms-rom} is similar to the NARMA model in \cite{Lu20Burgers}. The DD-VMS-ROM allows for fine interactions between the resolved scales through the $\wB$ matrix. In NARMA closure, these interactions are represented in a bundle $G(\ba)$ of polynomials of degree up to four, and this bundle has fewer parameters, allowing for time lags to represent the memory effects while avoiding overfitting.

\subsection{Stochastic DD-VMS-ROM by inference}\label{sec:S-ROM_framework}
Combining the strengths of both the DD-VMS-ROM and the inference-based approaches, we introduce a method to construct robust data-driven stochastic ROM (S-ROM) by inference. 

The premise is that the FOM \eqref{eq:FOM-Data} is subject to random initial conditions, which are sampled from a given probability measure $\mu$: 
\begin{equation}\label{eq:FOM_random_IC}
\begin{aligned} 
 \dot \bu= \bff (\bu), \quad \bu(0) = \bu_0 \sim \mu. 
\end{aligned}%
\end{equation}
While the method presented below is applicable to general nonlinear systems, to fix ideas, we assume that the nonlinear term  $\bff(\bu)$ is quadratic. For a given reduced dimension $r$, the S-ROM consists of a quadratic drift part and an additive noise term. Besides the Galerkin projection of $\bff(\bu)$, the drift part involves additional linear and bilinear terms whose coefficients are optimized based on multiple-trajectory FOM data.

More precisely, the proposed S-ROM fits the FOM data to the time series model 
\begin{equation}\label{eq:sROM}
 \ba(t_{l+1}) - \ba(t_l) \approx \left[(A + \wA)\ba + \ba^\top (B + \wB) \ba\right] \delta + \sqrt{\delta }\Sigma \bxi_l, \quad l=1,\ldots, n_l,
\end{equation}
where $\ba$ is $\mathbb{R}^r$-valued for a chosen reduced dimension $r$, the noise term $\{\bxi_l\}$ is an IID sequence of Gaussian random variables with distribution $N(0,Id)$, and $\{t_l=l\delta\}_{l=1}^{n_l}$ are the time instances with a time-step $\delta$ which can be relatively large (see Section \ref{sec:spaceTimeReduction} for an exploration). 
The parameters $\wA, \wB$ and $\Sigma$ are estimated by maximizing the likelihood of data consisting of multiple trajectories. We assume $\Sigma$ to be diagonal for computational efficiency. This model approximates the discrete-time flow map of the process $\ba$  in \eqref{eq:flowmap}. It is similar to an Euler-Maruyama discretization of the stochastic differential equation 
\begin{align}\label{2S-ROM-stoc}
\dot \ba 
= (A + \wA)\ba + \ba^\top (B + \wB) \ba + \Sigma \dot{\mathbf{W}}_t, 
\end{align}
where $\dot{\mathbf{W}}_t$ is white noise. 
This equation is a stochastic version of the 2S-DD-VMS-ROM in \eqref{eq:2S-ROM}.

The above S-ROM differs from the 2S-DD-VMS-ROM in three aspects: (i) its parameters $\wA, \wB$ fit multiple-trajectories instead of a single trajectory; (ii) it has an additional noise term to account for the uncertainty in the residual; (iii) it is a discrete-time model instead of a differential system and its parameters are time-step adaptive.  
As a result, the stochastic ROM can make predictions for new initial conditions sampled from the distribution and it is suitable for ensemble predictions with uncertainty quantification. In other words, it generates a new stochastic process approximating the original stochastic process of the FOM in distribution. Also, it accounts for the discretization error and can tolerate a larger time-step size. 

The construction of the stochastic ROM consists of three steps: data generation, extraction of basis functions and inference of the parameters from data. First, we generate multiple-trajectory data by the FOM, which resolves the system with high-resolution in space and time, with initial conditions sampled from $\mu$. Second, we construct basis functions from data, leading to a data-adaptive selection of hierarchical spaces. Third, we infer the parameters in the time series model \eqref{eq:sROM} from multiple trajectory data of $\ba$.

\paragraph{Data generation.} We generate data using the FOM with a time-step $\Delta t$. The data consists of many trajectories with the initial conditions sampled from the initial distribution: 
\begin{align} \label{eq:data}
\text{Data of multiple trajectories:}\quad  \{Y^{(m)} = \bu(t^F_{1:N_t})^{(m)} \in \R^{N_x\times N_t}\}_{m=1}^M, 
\end{align}
where $t^F_{1:N_t}= (t^F_1,\ldots,t^F_{N_t})$, with $t^F_l = l\Delta t$ being the time instances.  Hereafter, $\bu(t_l^F)\in \R^{N_x}$ denotes the FOM solution with $N_x$  space grid points at time $t^F_l$. These fine time instances will be downsampled to infer the S-ROM \eqref{eq:sROM} when a time-step $\delta$, larger than $\Delta t$, is used to construct the S-ROM. 

\paragraph{Extraction of POD basis and solution coefficients.} We construct the basis functions from data via POD and extract the coefficients $\ba$ by projecting the FOM solutions onto the POD basis.  
\begin{enumerate}
    \item (Extract basis functions by ensemble-POD): 
    We get basis functions by computing the eigenvalues and eigenvectors of the square matrix $\overline K_M$: 
\begin{align} \label{eq:Km}
\overline K_M = \frac{1}{M} \sum_{m=1}^M K^{(m)}, \quad \text{ with } K^{(m)} = \frac{1}{N_t} Y^{(m)}(Y^{(m)})^\top \in \R^{N_x \times N_x}. 
\end{align}
Here $K^{(m)}$ can be computed for each trajectory in parallel. We sort the eigenvalues in descending order, denote them by $\lambda_i$ and denote their eigenvectors by $\varphi_i\in \R^{N_x\times 1}$ for $i=1,\ldots,N_x$.  

\item  (Extract trajectory data of the low modes).  Projecting the solution $\bu$ to the POD basis: 
\[
\ba_i^{(m)} =  \varphi_i^\top Y^{(m)} \in \R^{1\times N_t}, \quad, i=1, \ldots, r;\, m= 1,\ldots, M, 
\]
we obtain the coefficients $\{\ba_i(t_l^F)= ( \bu(\cdot,t_l^F),\varphi_i(\cdot), i=1,\ldots, r \}_{l=0}^{N_t}$, which is a multivariate time series encoding the dynamics. 
\end{enumerate}

The above construction of the POD basis is almost the same as the classical approach of principle component analysis, which identifies the basis by the eigen--decomposition of the covariance matrix. The minor difference here is the time averaging step that discards the temporal correlations of the process $Y$, so that we can focus on the dominating modes of the dynamics. It is similar to the ensemble POD in \cite{gunzburger2017ensemble,gunzburger2020leray}, which uses multiple trajectories from systems with different parameters, in the sense that the initial condition can be viewed as a parameter of the system. The eigen--decomposition of $\overline K_M $ can be done directly (e.g., by Cholesky factorization) when its dimension is not too large, 
and more advanced methods (e.g., the nested-POD \cite{kadeethum2021non}) are available for higher dimensions.

\paragraph{Inference of the parameters.} 
 We estimate the parameters by maximizing the likelihood of data for the model \eqref{eq:sROM}. We start from downsampling the data trajectories in time to fit the time-step of the S-ROM. That is, each projected FOM trajectory $\ba(t^F_{1:N_t})$ with $t^F_j = j\Delta t$ is downsampled to $\ba(t_{1:n_t})$ with $t_l = l\delta$. In other words, if $\delta = {\rm Gap} \Delta t$, we make observations every $\rm Gap$ steps from the fine data. Denote the downsampled data by 
 \[
 \textbf{ Data }\{\ba^{(m)}(t_{1:n_t})\}_{m=1}^M, \, t_l = l\delta = l {\rm Gap} \Delta t. 
 \]

Since the S-ROM depends on the parameters linearly, the parameters $(\wA, \wB)$ are estimated by regression, which is equivalent to maximizing the likelihood of the data trajectories:  
 \begin{equation}\label{eq:MLE}
 (\hatA,\hatB)=\argmin{(\wA, \wB)} l(\wA, \wB): = \frac{1}{Mn_t}\sum_{m=1}^M \sum_{l=1}^{n_t}\|   F^{(m)}(t_{l})  - \wA \ba^{(m)}(t_l)  - (\ba^{(m)}(t_l))^\top  \wB \ba^{(m)}(t_l) \|^2, \end{equation} 
 where, with the notation $\Delta \ba(t_l) = \ba(t_{l+1}) - \ba(t_l) $,  the variable $F$ is defined by 
 \begin{align}\label{eq:F^m}
   F^{(m)}(t_{l}) & =  \frac{1}{\delta} \Delta \ba^{(m)}(t_{l}) - (A \ba^{(m)} + \ba^\top B \ba^{(m)})(t_{l}) \in \R^{r\times {1}}.  
 \end{align}
 The diagonal $\Sigma$ is estimated by the residual of the regression. 
 
 Note that $\wB \in \R^{r\times r\times r}$ is a 3D array such that for each $k$, the matrix $( \wB(i,i',k) )_{1\leq i, i'\leq r}$ is symmetric because it is the coefficient matrix of the terms $\ba_i\ba_{i'}$. Thus, including $\wA\in \R^{r\times r}$, we have  $n_r =r+ (r+1)r/2$ parameters to be estimated for each of the $r$ modes. To write in the form of least squares, we denote the estimator by $\bc =(\hatA(i,k)_{1\leq i, k\leq r},( \hatB(i,i',k) )_{1\leq i'\leq i \leq r, 1\leq k\leq r}) \in \R^{n_r\times r}$,  and denote the normal matrix and vector of regression by 
 \begin{equation}\label{eq:Abarbbar}
 \begin{aligned}
\bA^{(m)}(j,j') &= \frac{1}{n_t} \sum_{l=1}^{n_t} \psi_j(\ba^{(m)}(t_l))\psi_{j'}(\ba^{(m)}(t_l)), \,  1\leq j, j'\leq n_r ,\\
 \mathbf{b}^{(m)}(j) & = \frac{1}{n_t\delta} \sum_{l=1}^{n_t} \psi_j(\ba^{(m)}(t_l))F^{(m)}(t_l) \in \R^{r},  \,  1\leq j\leq n_r,
 \end{aligned}
 \end{equation}
 where $(\psi_j(\ba(t_l)), 1\leq j \leq n_r )$ is defined by 
  \begin{equation*}
  \psi_j(\ba(t_l)) = 
   \left\{
    \begin{array}{ll}
      \ba_{j}(t_l), & 1\leq j\leq r; \\[1ex]
      \ba_{i}(t_l)\ba_{i'}(t_l), &  j= i\times r + r' \text{ with }1\leq i'\leq i\leq r. 
          \end{array}
    \right.
  \end{equation*}
  Then, the estimator is solved by 
 \begin{equation}\label{eq:LSE}
 \bc_M = \bA_M^{-1}\mathbf{b}_M, \quad \text{ with }  \bA_M = \frac{1}{M}\sum_{m=1}^M \bA^{(m)},\quad \mathbf{b}_M = \frac{1}{M}\sum_{m=1}^M \mathbf{b}^{(m)}, 
 \end{equation}
where $\bA_M^{-1}$ denote the Moore--Penrose pseudo-inverse when the normal matrix is singular. In practice, when $\bA_M$ is ill-conditioned or singular, a regularization term often helps to lead to a robust estimator (see Section \ref{sec:Reg} for more details).  The residual provides us the estimator for $\Sigma$: 
\begin{equation}\label{eq:Sigma}
\Sigma_M = \mathrm{Diag}(\sigma(1), \ldots, \sigma(r)), \quad \text{ with } \sigma(k)^2 =  \|\bA_M\bc_M(\cdot,k) - \mathbf{b}_M(\cdot,k)\|^2_{\R^{n_r}}.
\end{equation}
  
 The resulted discrete-time S-ROM is designed to account for both the model error and discretization error. It has a consistent estimator (i.e., the estimator converges as the number of trajectories increases (see Theorem \ref{thm:conv} in the next section). Furthermore, it can tolerate a relatively large time step size \cite{CL15,Lu20Burgers,LiLuYe21}, and we will demonstrate it in Section \ref{sec:spaceTimeReduction} for a viscous Burgers equation.

\paragraph{Computational complexity.} At the learning stage, the major computation cost occurs when reading the FOM data twice: one for the construction of the POD basis $\{\varphi_k\}_{k=1}^r$ and the other for the computation of the coefficients $\{\ba(t_{1:n_t}))^{(m)}\}_{m=1}^M$. The rest of the computation uses only these coefficients, whose dimension is  significantly lower than the dimension of the FOM data. We implement all the trajectory-wise computation in parallel,
from the evaluation of $K^{(m)}$ in \eqref{eq:Km} to the estimator in \eqref{eq:LSE}. In particular, during the computation of the estimator, we first compute $\bA^{(m)}$ and $\mathbf{b}^{(m)}$ in \eqref{eq:Abarbbar} for each trajectory in parallel, then we assemble them as in \eqref{eq:LSE}.

\subsection{Convergence of the estimators}\label{sec:conv}
We show that the POD modes (the eigenvectors) and the parameter estimator converge as the number of trajectories increases. Such a convergence follows from the Central Limit Theorem, because the trajectories are independent and identically distributed with the randomness comes from the initial conditions.

The following assumption requires the solution field to be uniformly bounded almost surely, and it holds true for the weak solution of Equation \eqref{eq:VMS_full} with a wide range of nonlinearities and initial distributions, including the Burgers equation with random initial conditions to be studied in the next section. 

\noindent \textbf{Assumption A}: 
Assume that the FOM solution $\bu$ of \eqref{eq:FOM_random_IC}, defined for $t\in [0,T]$ and $x\in D\subset \R^d$ with $D$ being a bounded domain,  satisfies $\|\bu\|_\infty = \sup_{t\in[0,T],x\in D}|u(t,x)| < \infty$ for almost all initial conditions sampled from the initial distribution $\mu$.

The next theorem shows that the POD eigenvectors estimated in the previous section converge as the number of trajectories increases.  
\begin{theorem}[Convergence of POD eigenvectors]\label{thm:POD} 
The POD eigenvectors and eigenvalues converge as $M\to \infty$ under {\rm Assumption A}. More precisely, 
\begin{itemize}
\item Each eigenvalue of the averaged covariance matrix $\overline K_M $ in \eqref{eq:Km},  denoted by $\lambda_k^M$, converges to the corresponding eigenvalue $\lambda_k$ of $\E[\overline K_M ] =\lim_{M\to \infty} \overline K_M$, and $M^{-1/2} (\lambda_k^M-\lambda_k)$ is asymptotically normal. 
\item Each eigenvector of the eigenvalues with multiplicity one,  denoted by $\varphi_k^M$,  converges in $L^2(D)$ almost surely to the corresponding eigenvector $\varphi_k$ of  $\E[\overline K_M ]$, and  $M^{-1/2} (\varphi_k^M-\varphi_k)$ is asymptotically normal.  
\end{itemize}
\end{theorem}
\begin{proof} This is a classical result in principle component analysis, see e.g., \cite[Proposition 8-10]{dauxois1982asymptotic}.
\end{proof}

 \begin{remark} 
 When the data is continuous (i.e., not on the spatial grid points), the principal component analysis shows that similar convergence holds true for the eigenvalues and the eigenfunctions of the operators {\rm\cite{dauxois1982asymptotic}}. That is, let $K_M(x,y) =\frac{1}{M} \sum_{m=1}^M\frac{1}{T}\int_0^T  [u(t,x)u(t,y) ]dt$ and define the operator $L_{K_M}: L^2(D)\to L^2(D)$ by $L_{K_M}\varphi (x) = \int_D K_M(x,y)\varphi(y)dy$. Then, the eigenvalues and eigenfunctions of $L_{K_M}$ converge as $M$ increases. 
 \end{remark}

The next theorem shows that the estimator of $(\wA,\wB)$ is asymptotically normal as $M$ increases. 
\begin{theorem}[Convergence of parameter estimator]\label{thm:conv} 
Suppose that Assumption A holds true.  Suppose that 
 the smallest eigenvalue of $\bA_\infty = \lim_{M\to \infty}  \bA_M=  \E[\bA_M] $, the expectation of the random matrix $\bA^{(m)}$ in  \eqref{eq:Abarbbar},  is positive. 
Then, the estimator $ \bc_M $ in \eqref{eq:LSE} converges to $\bc_*=   \bA_\infty^{-1}\mathbf{b}_\infty$, where $\mathbf{b}_\infty=\E[ \mathbf{b}_M]$, almost surely and $M^{-1/2} (\bc_M - \bc_*)$ is asymptotically normal. 
\end{theorem}
\begin{proof}
Note that by the strong  Law of Large Numbers, $ \bA_M\to \bA_\infty$ and $\mathbf{b}_M\to  \mathbf{b}_\infty$ almost surely as $M\to \infty$. Thus, $ \bA_M^{-1} \to \bA_\infty^{-1}$ almost surely (using the fact that $A^{-1}-B^{-1} = A^{-1}(B-A)B^{-1}$). Then,  $\bc_M =\bA_M^{-1} \mathbf{b}_M \to \bc_*=\bA_\infty^{-1}\mathbf{b}_\infty$ almost surely, i.e. the estimator is consistent. Meanwhile, note that $\sqrt{M}(\mathbf{b}_M -\mathbf{b}_\infty)$ is asymptotically normal by the Central Limit Theorem. Together with the almost sure convergence of $\bA_M^{-1} $, we obtain the asymptotic normality of $M^{-1/2} (\bc_M - \bc_*)$. 
\end{proof}

We conclude this section by noting that the above convergence is in the number of independent trajectories, and it does not take into account of the length of the trajectories. While the length of the trajectories plays a limited role in the convergence in $M$, 
long trajectories help to identify the dominating modes for the longer-term dynamics. Also, the parameter estimator and the POD basis depend on the length of the data (in addition to the initial distribution).

\section{Numerical results for a viscous Burgers equation}
\label{sec:Burgers}

In this section, we perform a numerical investigation of the S-ROM for the one-dimensional viscous Burgers equation.
In Section~\ref{sec:Burgers-setup}, we present the mathematical and computational setups. 
Sections~\ref{sec:POD_estimation}-\ref{sec:convEstNum} examine the convergence of the POD basis and the parameters learnt from data. Section \ref{sec:SROM_performance} investigates the performance of the S-ROM in making predictions and quantifying the uncertainties. Finally, taking advantage of the S-ROM's features of being efficient and adaptive to spatial resolution and time-step, we explore optimal space-time reduction in Section \ref{sec:spaceTimeReduction}. 

\subsection{The one-dimensional viscous Burgers equation and the numerical setup}
\label{sec:Burgers-setup}

As an illustration of the stochastic DD-VMS-ROM framework presented in Section~\ref{sect_framework}, we consider the viscous Burgers equation posed on $(0, 1)$ and supplemented with homogeneous Dirichlet boundary conditions and a random initial condition:
\begin{equation} \label{eq:Burgers}
\begin{aligned} 
& u_{t} = \nu u_{xx} - uu_{x} \,, \quad 0 < x <1, \, t>0, \\
& u(0,t) = u(1,t) =0,   \quad  t \ge 0, \\
& u(\cdot, 0) = u_0(\cdot,\omega)\sim \mu.
\end{aligned}%
\end{equation}
Here  $u_0(\cdot,\omega)\sim \mu$ means that the initial condition is sampled from the measure $\mu$ on $L^2(0,1)$.  The randomness from initial conditions is important for data-driven modeling of the dynamics because it generates data trajectories that can sufficiently explore the dynamics of system, and it has been utilized in \cite{chorin_AveragingRenormalization2003,stinisMoriZwanzigReduced2012,choStatisticalAnalysis2014}. 

In the numerical experiments, we set the viscosity constant to be $\nu = 2\times10^{-3}$, and consider smooth random initial conditions in the form 
  \begin{align}  \label{eq:ic-random} 
  u_0(x,\omega) = \sum_{k=1}^K  \frac{w_k(\omega)}{k} \sin(\pi kx),
  \end{align} 
where we set $K=50$ to allow for fast oscillations in the sampled initial data profiles, and each $w_k$ is a random number sampled from the normal distribution with mean $\mu_u = 0.5$ and standard deviation $\sigma_u = 0.2$. A few typical such random initial conditions are shown in the left panel of Figure~\ref{fig:ICs}.
 
For each random initial condition, the initial boundary value problem \eqref{eq:Burgers} is solved with a finite element method, in which the spatial domain is discretized using piecewise linear finite elements with a uniform spatial mesh size $h=2^{-8}$, and the temporal discretization is performed by the implicit Euler method with a step size $\Delta t= 5\times 10^{-3}$.  The FOM solution is computed over the time window $[0, 2]$, leading to $N_t=401$ snapshots for each random initial condition. The number of spatial grid points is $N_x = 257$ for the chosen spatial resolution $h$. Thus, each FOM trajectory is stored in an $N_x\times {N_t}$ matrix, denoted by $Y^{(m)}$ for the $m$-th trajectory according to \eqref{eq:data}. A typical FOM solution is shown in the right panel of Figure~\ref{fig:ICs}.

\begin{figure}[htb]
\centering
    \includegraphics[width=0.85\textwidth]{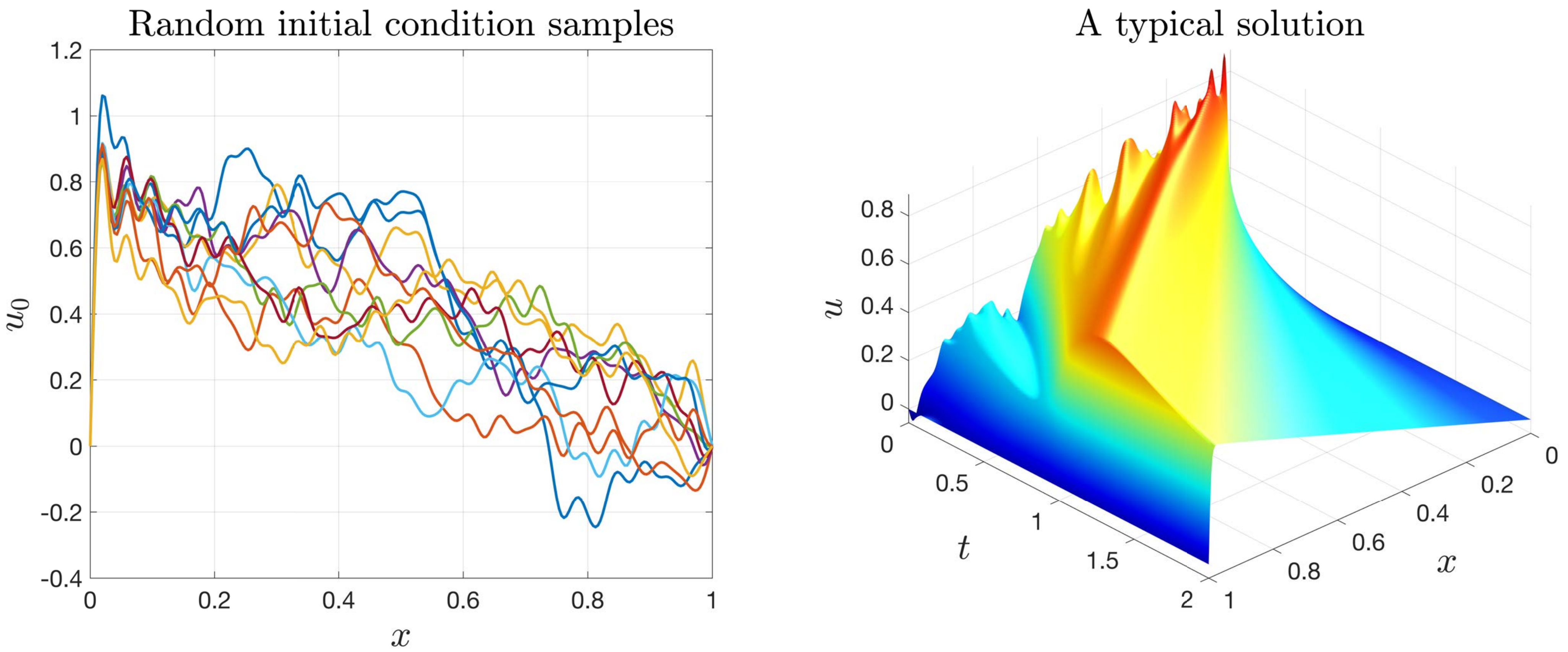}
    \caption{{\bf Left panel}: 10 random initial conditions sampled according to \eqref{eq:ic-random}. {\bf Right panel}: A solution of \eqref{eq:Burgers} over the time window $[0,2]$ with a random initial condition sampled from \eqref{eq:ic-random}.}
        \label{fig:ICs}
\end{figure}

We examine the construction of S-ROM of the form \eqref{2S-ROM-stoc} for the problem \eqref{eq:Burgers} in two groups of assessments. First, we investigate the convergence of the basis functions and the parameter estimator as the number of data trajectories increases. Second, we examine the performance of the S-ROM in two scenarios: trajectory-wise prediction and ensemble prediction. The trajectory-wise prediction aims to predict the deterministic dynamics, and the ensemble prediction aims to quantify the spread of the uncertainty from the unresolved modes. We will report the root mean square errors (RMSEs) of these predictions. In each simulation, we solve the estimated discrete-time S-ROM in \eqref{eq:sROM} exactly as it is:  we set $\Sigma=0$ when making trajectory-wise prediction, and we use the estimated $\Sigma$ with randomly sampled $\{\bxi_l\}$ when making ensemble prediction.

\subsection{POD Basis from multiple trajectories} \label{sec:POD_estimation}

We first check how the estimated POD modes and their corresponding eigenvalues stabilize as the number of training trajectories, $M$, increases. It turns out that the dominant POD modes estimated with a relatively small number of trajectories (e.g.~$M=20$) already capture qualitatively the shape of those obtained with significantly more trajectories. To facilitate a quantitative assessment, 
we denote the POD modes learnt from the first $M$ trajectories by $\{\varphi_j^M\}$, and we take those estimated with $\overline{M} = 1000$ as the reference. The numerics reveal that the first 10 POD modes, $\varphi_1^{\overline{M}}, \ldots, \varphi_{10}^{\overline{M}}$, already capture above $99.9\%$ of the averaged kinetic energy in each of the $1000$ training trajectories. We thus focus on these first 10 modes.   
\begin{figure}[htb]
\centering
   \includegraphics[width=0.95\textwidth]{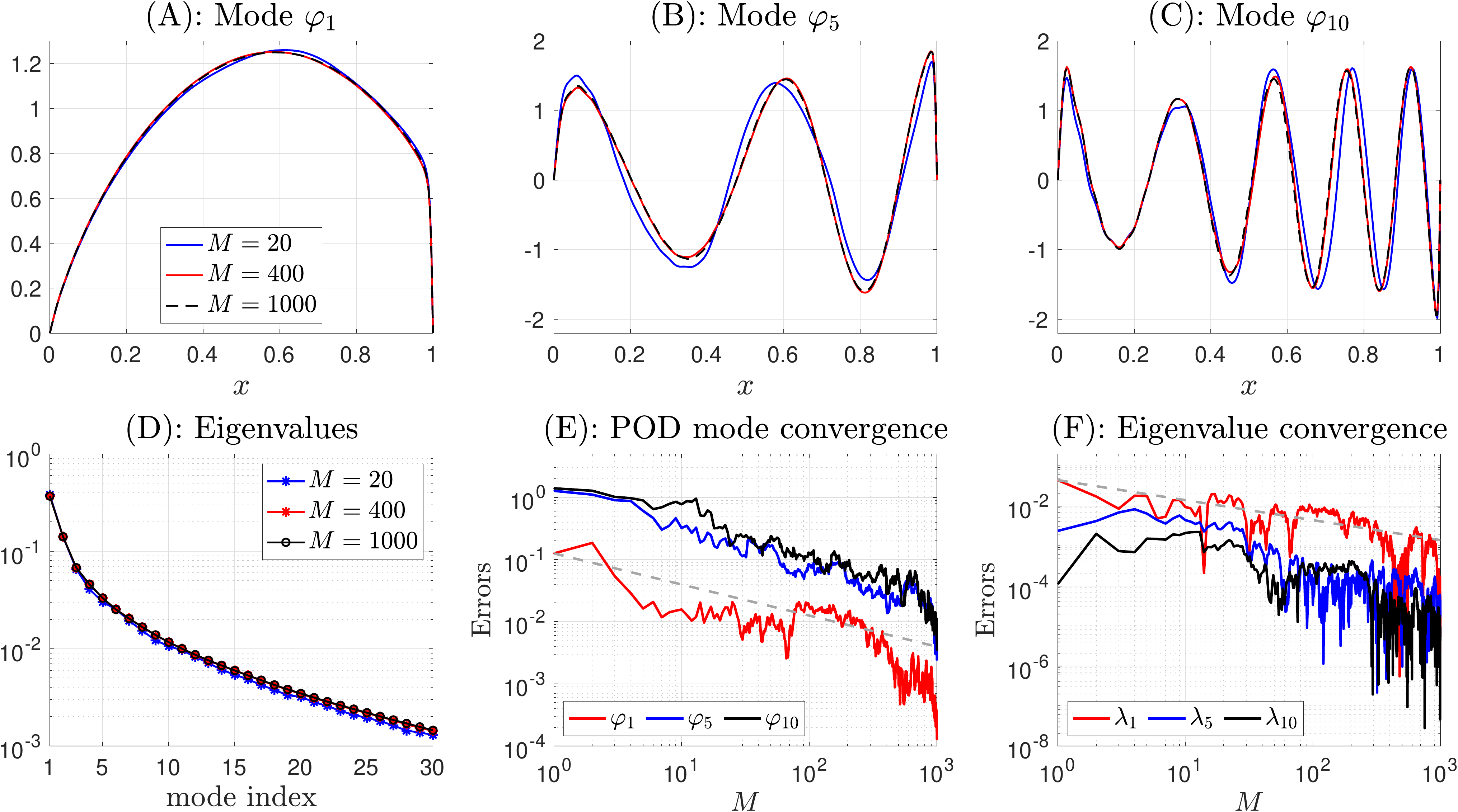}
    \caption{{\bf Panels (A-C)}: POD modes $\varphi_1$, $\varphi_5$ and $\varphi_{10}$, computed from the FOM solution datasets with $M=20$, $400$, and $1000$ trajectories, respectively. 
    {\bf Panels (D)}: The first 30 POD eigenvalues. 
    {\bf Panel (E)}:   $L^2$-error for the $j$-th POD mode, $\|\varphi^M_j - \varphi^{\overline{M}}_j\|_{L^2(0,1)}$, as the number of training trajectories $M$ increases towards $\overline{M}=1000$, shown for $j = 1, 5$, and $10$. {\bf Panel (F)}: Absolute error for the $j$-th POD eigenvalue, $|\lambda_j^M - \lambda_j^{\overline{M}}|$, as the number of training trajectories $M$ increases towards $\overline{M}=1000$, shown for $j = 1, 5$, and $10$. 
 The computed POD modes and their corresponding eigenvalues all show a clear trend of convergence as the number of training trajectories $M$ increases, at a rate close to the theoretical rate  $M^{-1/2}$ in Theorem \ref{thm:POD}, which is indicated by the dashed gray lines with slope $-0.5$. 
See Section \ref{sec:POD_estimation} for details.
    }
    \label{fig:pod}
\end{figure}

Panels (A)--(C) of Figure~\ref{fig:pod} show the estimated modes $\varphi_1,\varphi_{5}$, and $\varphi_{10}$ as $M$ increases. As can be observed, the estimated modes with $M=400$ are already almost indistinguishable with their respective reference ($\overline{M} = 1000$). The eigenvalues also stabilize quickly as shown in Panel (D) of Figure~\ref{fig:pod}. We also computed the $L^2$-error for each POD mode, $\|\varphi^M_j - \varphi^{\overline{M}}_j\|_{L^2(0,1)}$, as the number of training trajectories $M$ increases towards $\overline{M}$. The $L^2$-error for each mode follows a decreasing trend as $M$ increases with although local fluctuations; also, for each fixed $M$, the lower indexed modes have smaller errors. Panel (E) of Figure~\ref{fig:pod} show these $L^2$-errors for the modes $\varphi_1, \varphi_5$, and $\varphi_{10}$. 

A noteworthy feature of these $L^2$-errors is a power law decay with exponent $-0.5$ as marked by the dashed gray line in Panel (E) of Figure~\ref{fig:pod}. As shown in Theorem \ref{thm:POD}, such a power law decay with exponent $-0.5$ is just a manifestation of the Central Limit Theorem since the trajectories are independent and identically distributed with randomness from the initial condition. 
Also, in agreement with Theorem \ref{thm:POD}, this power law decay is also visible in the absolute error for the POD eigenvalues, $|\lambda_j^M - \lambda_j^{\overline{M}}|$, as shown in Panel (F) of Figure~\ref{fig:pod}, with an even more negative scaling exponent for some of the higher indexed modes.

\subsection{Convergence of parameter estimator}\label{sec:convEstNum}
We show next that the parameter estimator converge as the number of trajectory increases. Meanwhile, we show that the trajectory-wise estimator can vary largely between trajectories. Thus, it is important to estimate the parameters using multiple trajectories when constructing ROM for predictions from different initial conditions, particularly for random or stochastic systems.

\begin{figure}[htb]
\centering
    \includegraphics[width=1\textwidth]{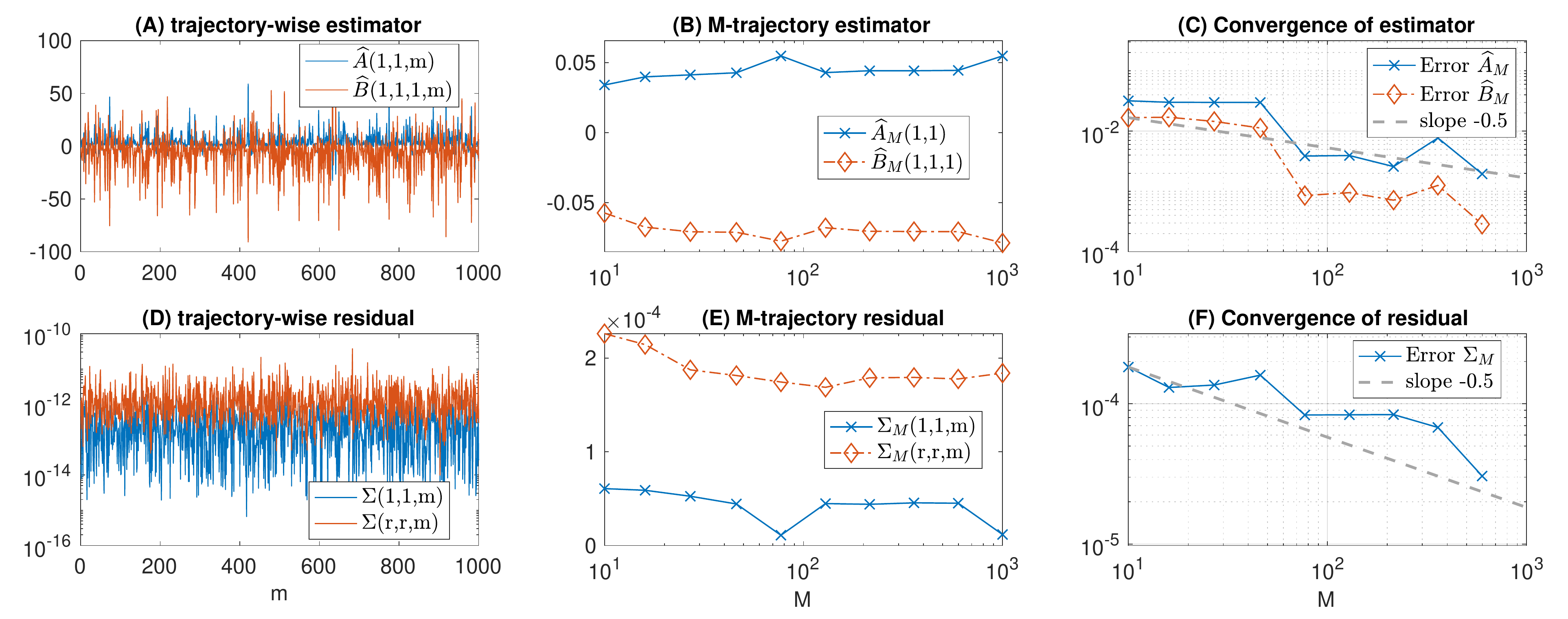}
    \caption{Convergence of parameter estimator. \textbf{Panel (A)} shows that the single-trajectory estimators vary largely between trajectories, and \textbf{Panel (D)} shows that the corresponding residuals are extremely small. Together, they indicate an overfitting by the single-trajectory ROM. On the contrary, \textbf{Panel (B) and (E) } show that the M-trajectory estimators stabilize quickly as $M$ increases. \textbf{Panel (C) and (F)} show the errors of the estimators in comparison to the reference estimator with $\overline{M} =1000$. The errors decay at a rate close to the theoretical rate $M^{-0.5}$ in Theorem \ref{thm:conv}.  See Section \ref{sec:convEstNum} for details. 
    } \label{fig:estConv}
\end{figure}

Figure \ref{fig:estConv} shows the estimators of the S-ROM with $r=10$ POD modes (those estimated from 1000 trajectories in the previous section) and with a time gap $\rm Gap =5$. Panel (A) shows that the trajectory-wise estimators of $\wA(1,1)$ and $\wB(1,1,1)$ vary largely between trajectories, and Panel (D) shows that the residuals of the trajectory-wise estimators are extremely small, indicating an overfitting. Together they show that while a trajectory-wise ROM can fit each specific trajectory data well with a negligible residual, it is sensitive to the data and is not generalizable. On the other hand, Panel (B) shows that the $M$-trajectory estimators of $\wA(1,1)$ and $\wB(1,1,1)$ stabilizes very quickly: as can be observed, the estimators with $M= 10$ are already close to those with $M=1000$ trajectories. Similar convergence trend is observed for the corresponding residuals in Panel (E). We remark that such a quick stabilization is partially due to that the POD modes are estimated from 1000 trajectories for all the estimators; if the POD modes are estimated for each sample size, which will oscillate as we have seen in Figure \ref{fig:pod}, we expect a sightly slower stabilization.    

Figure \ref{fig:estConv} (C) and (F) further show that the estimators converge at a rate close to $M^{-0.5}$, in agreement with Theorem \ref{thm:conv}. Here we compute the errors of the estimators in comparison to the reference estimators using $\overline{M}=1000$ trajectories. That is, for each of $M\in \{\floor{10^{1+j\alpha}}, j=0,\ldots,9\}$ with $\alpha = 2/9$,  the errors of $\widehat A_M $,  $\widehat B_M $ and  $\Sigma_M $ are computed by the Frobenius norms defined by 
\begin{equation*}
\begin{aligned}  
\| \widehat A_M - \widehat A_{\overline{M}}\|^2 & = \frac{1}{r^2}\sum_{1\leq i,k\leq r}|\widehat A_M (i,k) - \widehat A_{\overline{M}}(i,k)|^2, \\
\| \widehat B_M - \widehat B_{\overline{M}}\|^2 & = \frac{2}{r^2(r+1)}\sum_{1\leq i\leq i'\leq r,1\leq k\leq r}|\widehat B_M (i,i',k) - \widehat B_{\overline{M}}(i,i',k)|^2, \\
\| \Sigma_M - \Sigma_{\overline{M}}\|^2 & = \frac{1}{r}\sum_{1\leq k\leq r}|\Sigma_M (k,k) - \Sigma_{\overline{M}}(k,k)|^2 .
\end{aligned}
\end{equation*}
  
The parameters of S-ROM estimated from $\overline{M}$  trajectories are presented in Appendix \ref{append:paraEst}.

\subsection{S-ROM performance}\label{sec:SROM_performance}
We examine the performance of the S-ROM in two scenarios: deterministic single trajectory prediction and stochastic ensemble prediction. In the trajectory-wise prediction, the S-ROM makes a single trajectory prediction for each given initial condition by setting the stochastic force to be zero (i.e., $\Sigma=0$). We compare a typical solution field of S-ROM with those of the $r$-mode projection of the FOM solution and G-ROM (the Galerkin ROM with $r$-modes, i.e., $\dot \ba
= A \ba + \ba^\top B \ba$). We also report the statistics of the root mean square errors (RMSE) of the S-ROM to the $r$-mode projection of the FOM in multiple predictions. In the stochastic ensemble prediction, we turn on the stochastic force in the S-ROM and generate an ensemble of trajectories for each initial condition. 
The ensemble represents the spread of the uncertainty from the unresolved scales (i.e., the conditional distribution of the process), which is important for data assimilation \cite{chen2022conditional,LTC17}. 

The numerical settings are as follows. The S-ROM and G-ROM have $r=10$ POD modes (since it captures $99.9\%$ of the averaged kinetic energy of almost each trajectory, see Section \ref{sec:POD_estimation}) and a time step $0.025$ (i.e., with time gap $\rm =5$). The S-ROM uses the parameters estimated from Section \ref{sec:convEstNum}. The solutions are on the time interval $[0,4]$, twice the length of the training time interval. Their initial condition is the $r$-mode projection of the FOM's initial condition.

\begin{figure}[htb]
\centering
   \includegraphics[width=0.98\textwidth]{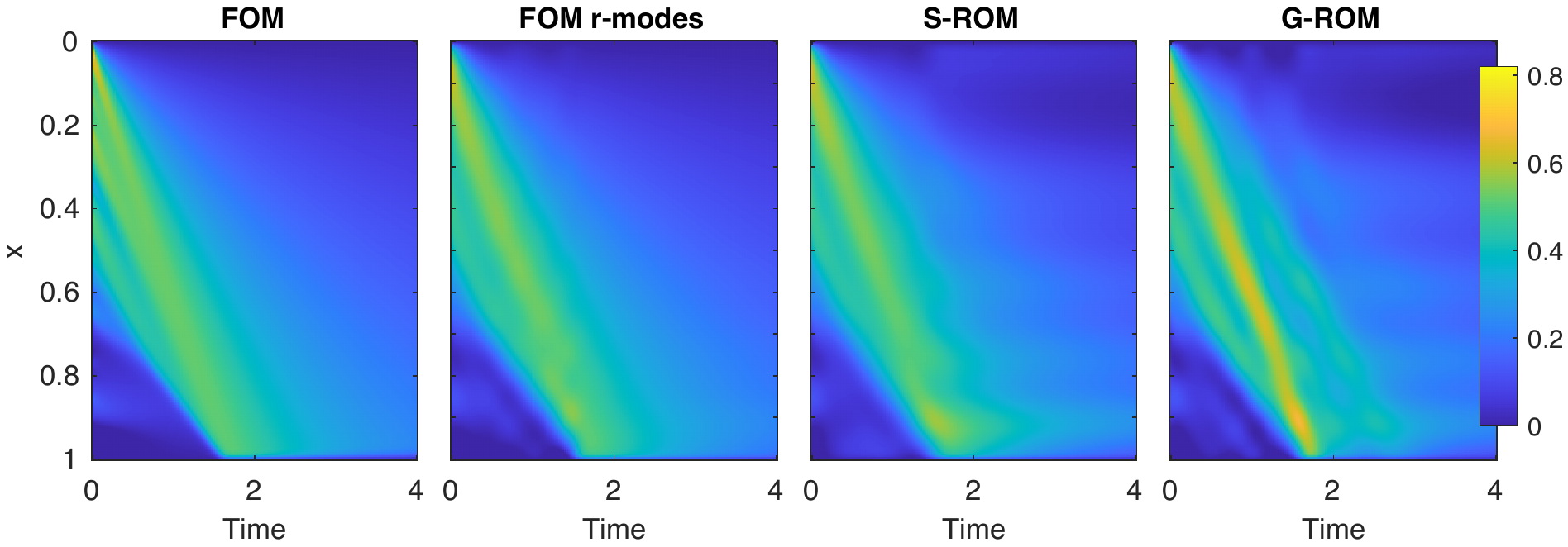}
    \includegraphics[width=0.98\textwidth]{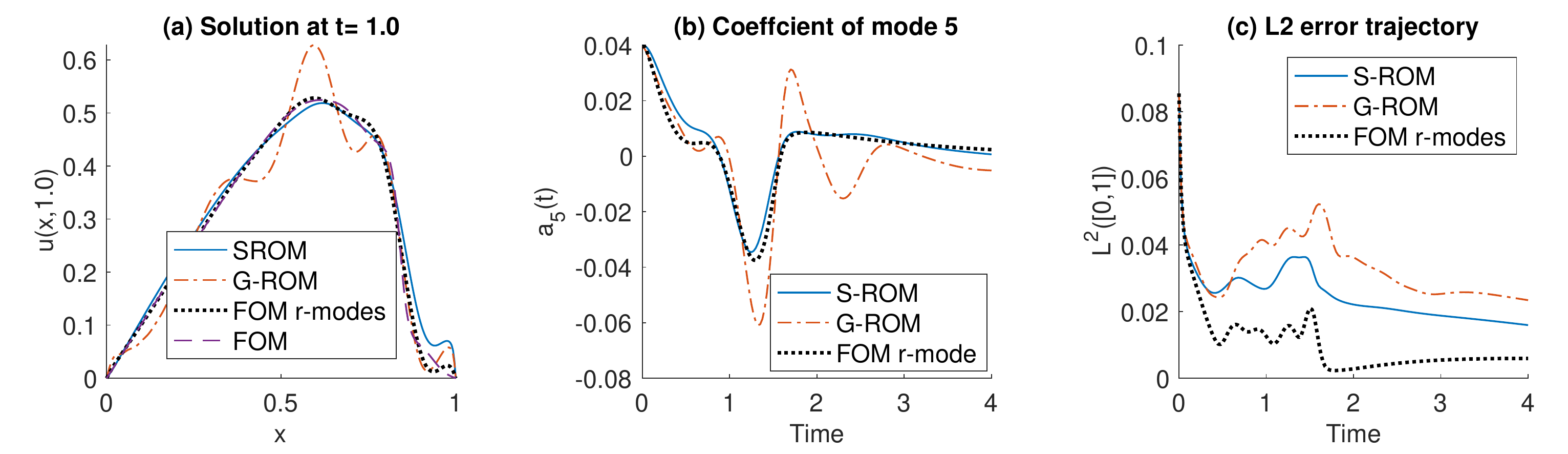} \vspace{-3mm}
    \caption{A typical solution trajectory of the S-ROM, in comparison with the FOM, the $r$-mode projection of the FOM, the G-ROM solutions. The \textbf{top row} compares the four solution fields. The 
\textbf{bottom row} compares the solutions at time $t=1.0$ in \textbf{(a)}, the trajectory of $\ba_5(t)$ in \textbf{(b)} and the trajectory of $L^2(0,1)$ errors, in comparison to the FOM solution in \textbf{(c)}. As can be seen, the S-ROM solution is more accurate than the G-ROM in approximating the FOM. 
     }
    \label{fig:1traj}
\end{figure}

\paragraph{Deterministic single trajectory prediction}
Figure \ref{fig:1traj} compares a typical solution field of S-ROM with those of the FOM, the $r$-mode projection of FOM (denoted by ``FOM $r$-modes''), and the G-ROM. It shows that the S-ROM solution is more accurate than the G-ROM in approximating the $r$-mode projection of the FOM. More specifically, the top row shows that the S-ROM has a solution field significantly closer to the $r$-mode projection of the FOM than the G-ROM's. 
The bottom row provides detailed profiles of the solution field: the spatial profile at time $t=1.0$ in \textbf{(a)} and the trajectory of  $\ba_5(t)$ in \textbf{(b)}. The spatial curves of $u(x,1.0)$ in  \textbf{(a)} shows that the S-ROM is significantly closer to the FOM's $r$-mode projection than the G-ROM. Similar superior performance is observed in \textbf{(b)}, which presents the trajectory of $\ba_5(t)$. Furthermore, \textbf{(c)} shows that S-ROM has $L^2(0,1)$ errors smaller than those of the G-ROM. Here the $L^2(0,1)$ error is computed as $\|\widehat u_r(\cdot,t) - u(\cdot,t)\|^2_{L^2(0,1)}$, with $\widehat u_r$ coming from the S-ROM, the G-ROM and the $r$-mode projection of the FOM. We note that the $r$-mode projection of the FOM has non-negligible errors because of the missing higher modes.

\begin{figure}[H]
\centering
   \includegraphics[width=0.98\textwidth]{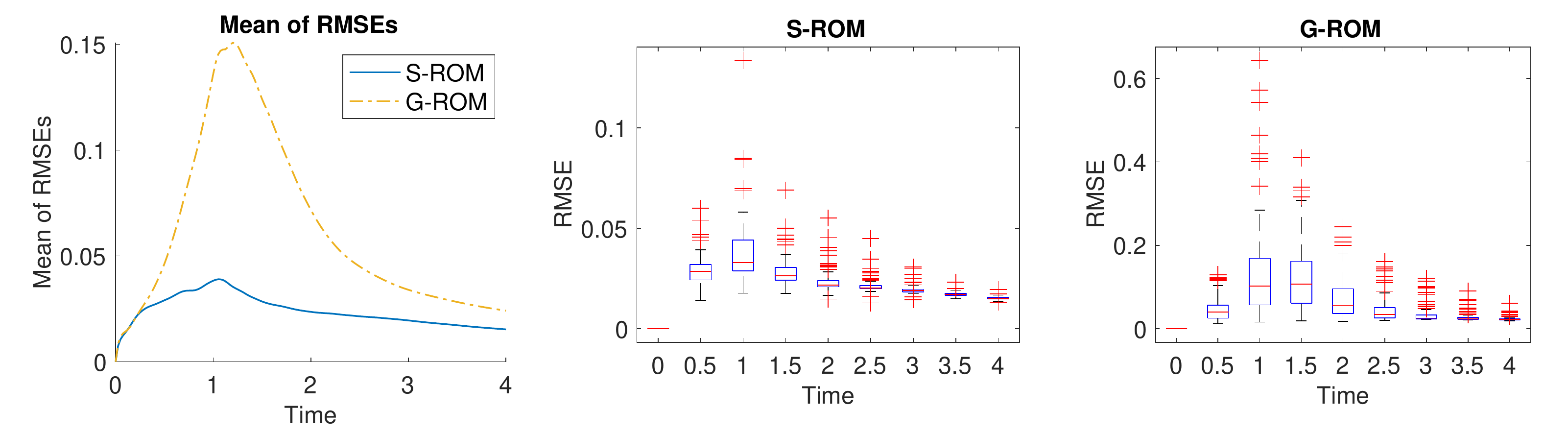}
    \caption{ RMSEs of 100 prediction trajectories by the S-ROM and G-ROM. \textbf{Left:} the mean of the RMSEs;  \textbf{Middle and Right:} the box-plots of the RMSEs in the time interval $[0,4]$. For each box, the central mark indicates the median, and the bottom and top edges of the box indicate the $25$-th and $75$-th percentiles, respectively. The whiskers extend to the most extreme data points without outliers, and the outliers are plotted individually using the ``$+$'' marker symbol. The S-ROM has RMSEs significantly smaller than those of the G-ROM. }
    \label{fig:boxplots}
\end{figure}

We further examine the superior performance of S-ROM over G-ROM in the prediction with 100 new initial conditions, and report the statistics of the RMSEs of the trajectories in Figure \ref{fig:boxplots}. The RMSEs are computed as 
\begin{equation}\label{eq:rmse}
{\rm RMSE(t) } = \| \widehat \ba(t) - \ba(t) \| = (\sum_{j=1}^r | \widehat \ba_j(t) -  \ba_j(t) |^2)^{1/2},
\end{equation}
for $t\in [0,4]$ for each trajectory, where $ \widehat \ba(t)$ comes from the S-ROM or G-ROM, and $\ba(t)$ is the $r$-mode projection of the FOM. These plots clearly show the improvement brought by the S-ROM: (1) the S-ROM's median RMSEs are less than 0.04, while those of G-ROM are about 0.15, about three times larger. Similar improvements are observed for the $25$-th and the $75$-th percentiles. (2) the outliers of S-ROM are less than 0.15, while those of the G-ROM exceeds 0.6. Thus, the S-ROM is significantly more accurate than the G-ROM in approximating the $r$-mode projection of the FOM.

\begin{figure}[htb]
\centering
   \includegraphics[width=0.8\textwidth]{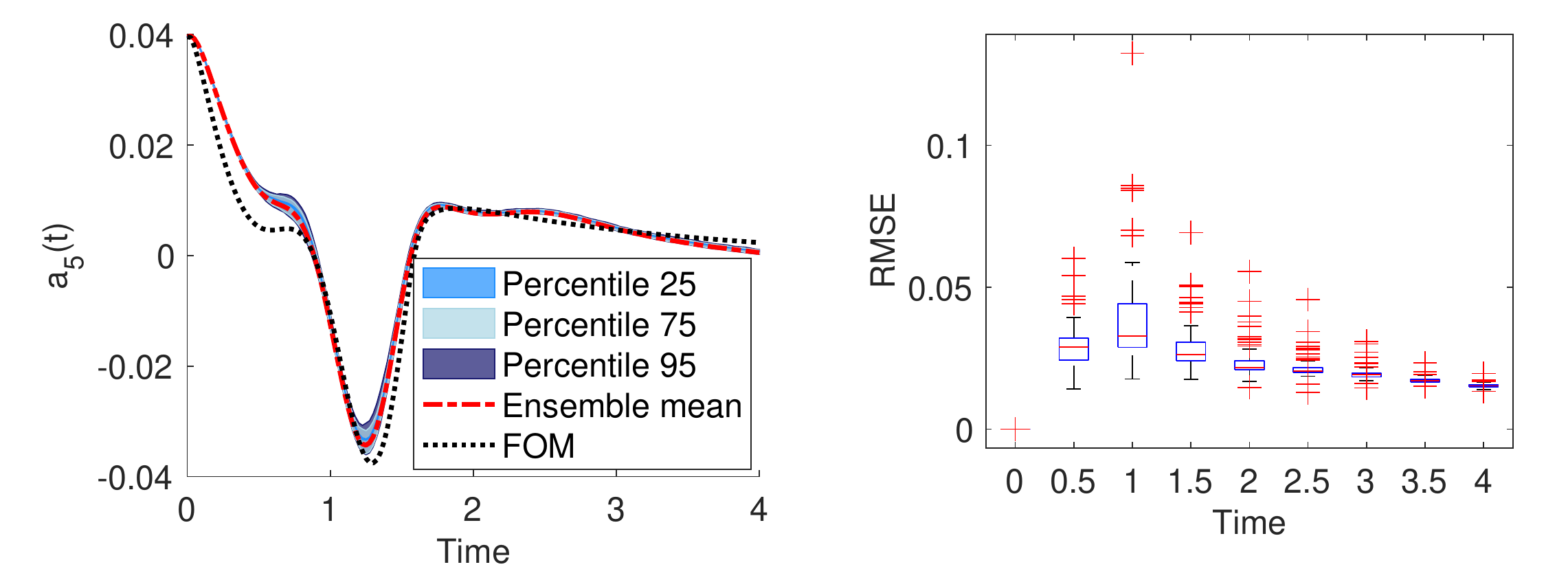}
    \caption{ Stochastic ensemble prediction. \textbf{Left:} An ensemble of 100 random trajectories of the S-ROM from an initial condition, as well as its ensemble percentiles (25,75 and 95) and mean (in red color), in comparison with the FOM's trajectory. \textbf{Right:} the box-plot of the RMSEs of the ensemble mean (as shown in left) in 100 independent simulations.  }
    \label{fig:stocEns}
\end{figure}

\paragraph{Stochastic ensemble prediction.} 
Next, we investigate the stochastic ensemble prediction by the S-ROM from an initial condition when the stochastic force is turned on. This is the setting when we have only the partial observation of the $r$-model of the FOM, and we would like to predict the future dynamics without resolving the FOM. In this setting, the S-ROM ensemble is the output corresponding to the realizations of the stochastic force. The ensemble represents the spread of the uncertainty from the unresolved scales, which is important for data assimilation \cite{chen2022conditional,LTC17}.  The left panel in Figure \ref{fig:stocEns} shows the ensemble with 100 trajectories of $\ba_t(t)$ (in cyan color) from a fixed initial condition, along with the ensemble mean and the trajectory of the FOM. It can be seen that the ensemble spreads out in the time interval $[0,1]$ and the spread is the widest when the trajectory changes convexity, which is due to nonlinear interaction between modes; and the spread becomes narrow after $t=2$ when the shock is almost formulated (see Figure \ref{fig:1traj} top row and Figure \ref{fig:ICs} right panel). The right panel of Figure \ref{fig:stocEns} further shows the statistics of the RMSEs of the ensemble mean in 100 independent such simulations (that is, in each simulation, we generate an ensemble of 100 prediction trajectories, obtain the ensemble mean, and then compute the RMSE of the ensemble mean as in \eqref{eq:rmse}). The statistics of the RMSEs are similar to those of deterministic prediction Figure \ref{fig:boxplots} (middle), reflecting the fact that the noise strength parameter $\Sigma$ is very small (see Figure \ref{fig:estConv}) compared with other parts of the vector field in the ROM.

\subsection{Discussion on space-time reduction}\label{sec:spaceTimeReduction}
The most attractive feature of the S-ROM is that it is adaptive to the resolution $r$ and time step size $\delta$. In the previous section, we focused on demonstrating the performance of S-ROM with $r=10$ and $\delta=0.025$ (i.e., time gap $=5$). A natural question is that when $(r,\delta)$ changes, how will the S-ROM adapt. In particular, for a given resolution $r$, what is the maximal time step size that the S-ROM can remain stable and what is the optimal time step size such that the S-ROM makes the most accurate predictions. 

\begin{figure}[htb]
\centering
   \includegraphics[width=0.5\textwidth]{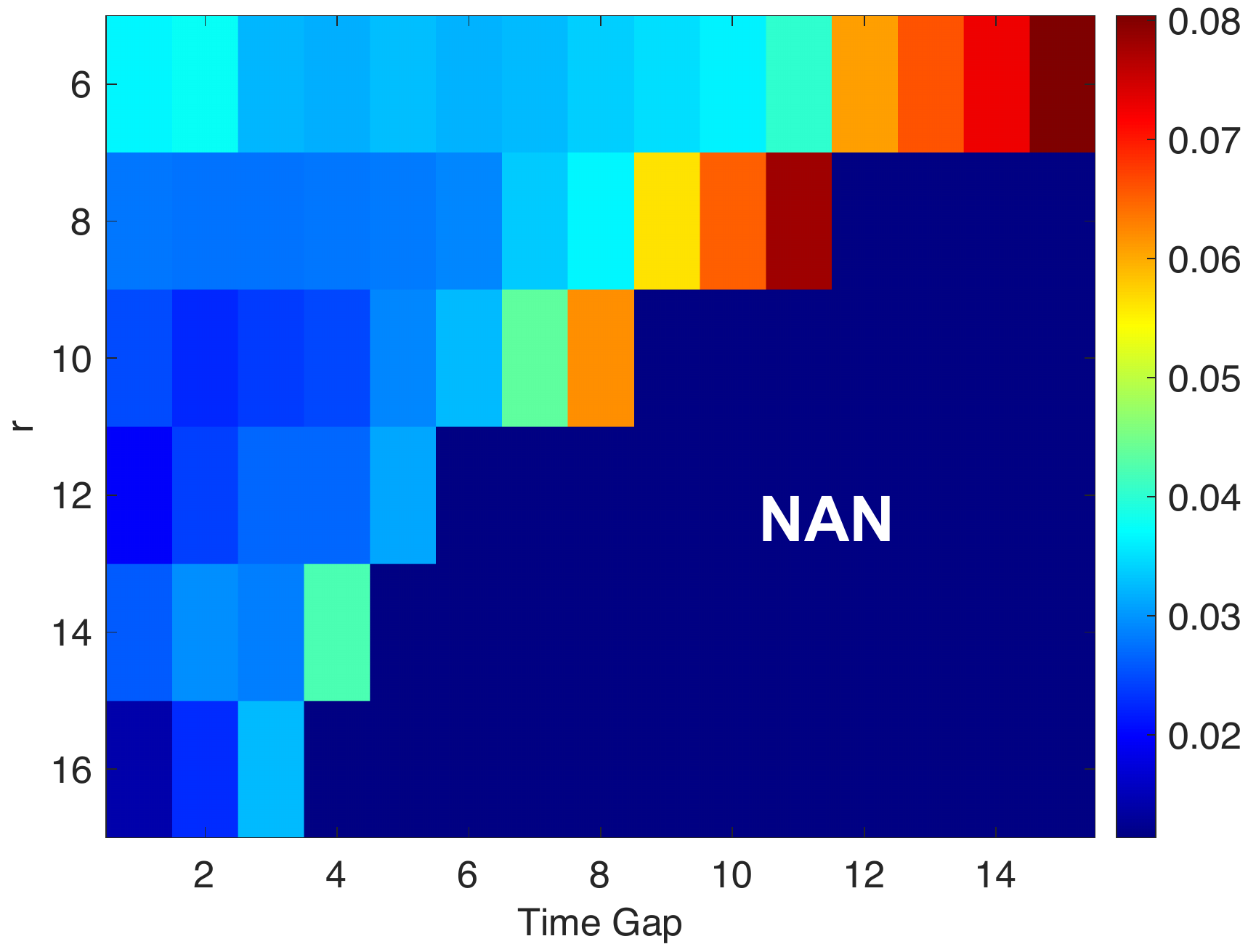}
    \caption{Average RMSEs of trajectory prediction from 90 S-ROMs with resolutions $r\in \{6,8,\ldots,16\}$ and time step sizes $\delta = {\rm Gap}\times 0.005$ with ${\rm Gap}\in \{1,\ldots,15\}$. 
    For each pair $(r,{\rm Gap})$, an S-ROM is estimated from data consisting of 1000 trajectories on time interval $[0,2]$, and it is used to make predictions on the time interval $[0,4]$ from 200 random initial conditions; the average RMSE is the time-average of the RMSEs of these 200 predicted trajectories. The darkest blue color with the ``NAN'' marker indicates that the S-ROM has at least one blowup solution among the 200 predictions. As $r$ increases, the maximal time step size of stability decreases, from ${\rm Gap}=11$ for $r=8$ to ${\rm Gap} = 3$ for $r=16$, due to the increasing stiffness. Note that for $r\in \{6,8,10\}$, the average RMSE first decreases then increases as the time gap increases, thus, suggesting that the S-ROM makes the most accurate prediction when the time step size is medium large. 
 }
    \label{fig:spacetime}
\end{figure}
To answer these questions, we test 90 S-ROMs with $r\in \{6,8,10,12,14,16\} $ and with $\delta = {\rm Gap}\times 0.005$ with ${\rm Gap} \in \{1,\ldots,15\}$.
 All the S-ROMs are trained from the dataset consisting of 1000 data trajectories, downsampled with time gap from the FOM solution with time step size $\Delta t= 0.005$. 
 For each $(r,\delta)$, we estimate an S-ROM, compute the mean RMSEs of 200 deterministic predictions by the estimated S-ROM, and report the time-average of the mean RMSE trajectory in Figure \ref{fig:spacetime}. Here the darkest blue color with the ``NAN'' marker indicates that the S-ROM has at least one blowup solution among 200 prediction trajectories.  As $r$ increases, the stiffness of the S-ROM increases, and its maximal time step size of stability decreases, from ${\rm Gap}=11$ for $r=8$ to ${\rm Gap} = 3$ for $r=16$. When the resolution is $r=6$, the S-ROMs with all the time gaps are stable, but the average RMSE first decreases then increases as the time gap increases. Such a pattern is also observed for the S-ROMs with $r=8$ and $r=10$. This pattern suggests that the S-ROM makes the most accurate prediction when the time step size is medium large. 
 
In summary, the S-ROMs are adaptive to resolution and time-step size. Its performance is best when there is a balanced space-time reduction. It remains open to understand the mechanism of such a balanced space-time reduction, and its connection with numerical error and statistical learning.


 \section{Conclusions} \label{sec:conclusion}
 
 We have proposed and investigated a parametric inference approach to construct data-driven stochastic reduced order models (S-ROM) that combines physical insights with data. The S-ROM is trained from multiple-trajectory data with random initial conditions and it is robust to make accurate predictions from new initial conditions. The framework consists of two stages. First, it constructs basis function from data by POD. Then, it infers a closure model for the coefficients of the
 first POD modes from the projected data trajectories. The current study focuses on systems with quadratic nonlinearities and constructs closure terms with linear and quadratic terms. We prove that the estimated POD basis and the parameters in the S-ROM converge at a rate $M^{-1/2}$, with $M$ being the number of data trajectories.  

We demonstrate the performance of the S-ROM on a viscous Burgers equation with random initial conditions. Numerical results verify the convergence of the POD basis and the parameters. They also confirm the superior performance of the S-ROM in making accurate predictions and representing the spread of the uncertainty due to the unresolved scales. Furthermore, the S-ROM is adaptive to the spatial resolution and time-step size. For a given spatial resolution, we 
demonstrate that the S-ROM makes the most accurate prediction when the time step size is medium large. 

We conclude by reiterating some key features of the novel S-ROM framework 
and the issues to be further investigated in future work. The framework of constructing S-ROM via  discrete-time flow map approximation has four major advantages:   \vspace{-2mm}
\begin{enumerate}[leftmargin=*]\setlength\itemsep{-1mm}
\item It is applicable to general data-driven models for high-dimensional dynamical systems, either random or stochastic. In particular, when there are physical insights to provide parameter forms, the ROM can be efficiently trained by regression with proper regularization to achieve performance guarantee.  
\item The reduction is in both space and time. Thus, the S-ROM achieves efficient simulation that can be orders of magnitudes faster than the FOM. 
\item The training can be done using only short trajectories for autonomous systems, and the constructed S-ROM can make predictions far beyond the training time interval, because the S-ROM learns a flow map that does not vary in time. Furthermore, the S-ROM can make prediction for new initial conditions sampled from the initial distribution. 
\item With a small noise term, the S-ROM captures the short-term dynamics of the resolved modes by its deterministic parts, and
can represent the spread of the uncertainty due to unresolved scales. Thus, the S-ROM is well-suited for predictive modeling that requires uncertainty quantification, e.g., data assimilation. 
\end{enumerate}

There are also a few aspects to be further investigated:
\vspace{-2mm}
\begin{enumerate}[leftmargin=*]\setlength\itemsep{-1mm}
\item The constructed S-ROM depends on the initial distribution through the training data. The reason is that, although the flow map of the FOM (the solution operator) is independent of the initial condition,
the flow map of the ROM variables depends on the initial condition of the unresolved variables. 
Thus, when the initial distribution changes, the S-ROM must be trained again. It is of interest to understand how the S-ROM parameters depend on the initial distribution. 
\item  It remains open to understand the mechanism of the balanced space-time reduction that achieves the most accurate prediction. This question is in the same spirit as the bias-variance tradeoff in statistical learning theory \cite{CuckerSamle02}. 
\end{enumerate}


\appendix
\section{Regularization}\label{sec:Reg}
Regularization plays an important role in the estimation of the parameter in S-ROM, because the normal matrix in the regression is often ill-conditioned or even singular \cite{iliescu2018regularized}. 
In our numerical tests, the normal matrix $\bA_M$ in \eqref{eq:LSE} is nonsingular but mostly ill-conditioned, with conditional numbers mostly in the range from $10^6$ to $10^{12}$ when $M>10$. Thus, the parameter estimation by solving the linear equation $\bA_M \bc = \mathbf{b}_M$ in \eqref{eq:LSE} is an ill-posed inverse problem. Then, regularization is necessary to avoid the numerical and model errors in $\mathbf{b}_M$ being overly amplified. We regularize the problem using the Euclidean norm. That is, we minimize the regularized loss function 
\[
\mE_\lambda(\bc) = \mE(\bc) + \lambda \|\bc\|^2, \quad \text{ with } \mE(\bc) = \bc^\top \bA_M \bc - 2\mathbf{b}_M^\top \bc 
\]
where $\|\cdot\|$ is the Euclidean norm, and the $ \mE(\bc) $ comes from the loss function in \eqref{eq:MLE}. We select the optimal regularization parameter $\lambda$ by the L-curve method (see \cite{hansen_LcurveIts_a}, and we refer to \cite{gazzola2019ir,LLA22} for recent developments) as follows. First, we solve $c_\lambda= (\bA_M+\lambda I)^{-1} \mathbf{b}_M$ by the minimum norm least square solution. Then, we find the regularization parameter that maximizes the curvature of the curve:
\[ l(\lambda) = (x(\lambda), y(\lambda)) := (\log(\mE(\bc_\lambda), \log(\|\bc_\lambda\|)).\]
 Recall that the curvature of $l$ is $\kappa(\lambda)= \frac{x'y'' - x' y''}{(x'\,^2 + y'\,^2)^{3/2}}$. Thus, we compute the curvature by central difference approximation of these derivatives with a mesh for $\lambda$ between the minimal and maximal eigenvalues of $\bA_M$, and we select the $\lambda$ with the maximal curvature.

 \section{Parameters in the S-ROM model}\label{append:paraEst}
Table \ref{tab:Atilde} and Figure \ref{fig:Btilde} show the estimated $\wA$ and $\wB$ of S-ROM estimated from $\overline{M}=1000$ trajectories. This S-ROM has $r=10$ POD modes, and its time step is $\delta = 0.025$ (i.e., its time gap is $\rm Gap =5$). We present the parameter  $\wA \in \R^{r\times r}$ in Table \ref{tab:Atilde} and present $\wB\in \R^{r\times r\times r}$ in scaled images in Figure \ref{fig:Btilde}.  

\begin{table}[H]\vspace{-3mm}
\centering
\caption{Parameter $\wA$ in S-ROM with $r=10$ and $\delta = 0.025$, estimated from $\overline{M}=1000$ trajectories.}\label{tab:Atilde}
\begin{tabular}{c | cccc cccc cc }
 $j$ & 1 & 2 & 3& 4 & 5 &6 &7 & 8  & 9 & 10 \\
\hline
$\wA(1,j)$& 0.05  & -0.03  & -0.05  & -0.11  & 0.14  & -0.13  & 0.15  & -0.16  & 0.17  & -0.18  \\ 
$\wA(2,j)$& -0.08  & 0.01  & 0.08  & 0.11  & -0.21  & 0.17  & -0.25  & 0.25  & -0.30  & 0.25  \\ 
$\wA(3,j)$& -0.10  & 0.00  & -0.02  & 0.05  & -0.18  & 0.22  & -0.28  & 0.31  & -0.40  & 0.34  \\ 
$\wA(4,j)$& -0.06  & 0.00  & 0.02  & -0.15  & -0.05  & 0.19  & -0.09  & 0.26  & -0.30  & 0.36  \\ 
$\wA(5,j)$& 0.07  & -0.00  & -0.01  & -0.06  & -0.30  & 0.03  & 0.29  & -0.14  & 0.36  & -0.10  \\ 
$\wA(6,j)$& -0.04  & 0.00  & 0.00  & 0.08  & 0.01  & -0.59  & 0.19  & 0.31  & -0.02  & 0.36  \\ 
$\wA(7,j)$& 0.07  & -0.00  & -0.00  & -0.04  & 0.22  & 0.02  & -0.98  & 0.32  & 0.45  & -0.03  \\ 
$\wA(8,j)$& -0.06  & 0.00  & 0.00  & 0.07  & -0.07  & 0.30  & 0.18  & -1.57  & 0.65  & 0.80  \\ 
$\wA(9,j)$& 0.07  & 0.00  & -0.00  & -0.01  & 0.15  & 0.03  & 0.45  & 0.47  & -2.50  & 1.75  \\ 
$\wA(10,j)$& -0.06  & 0.00  & 0.00  & 0.04  & -0.03  & 0.30  & -0.03  & 0.80  & 1.24  & -3.39  \\ 
   \end{tabular}
   \end{table}
\vspace{-3mm}

   \begin{figure}[H] \vspace{-3mm}
\centering
    \includegraphics[width=1\textwidth]{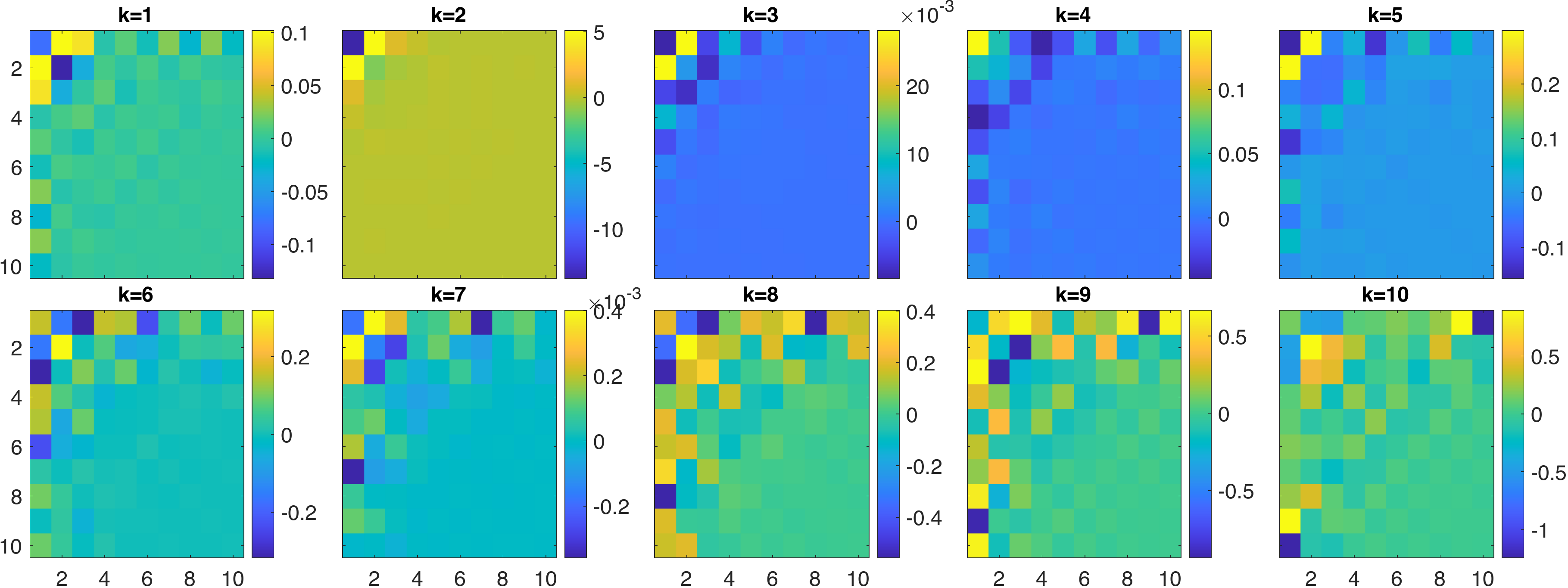}
\caption{Parameter $\wB$ in S-ROM with $r=10$ and $\delta = 0.025$, estimated from $\overline{M}=1000$ trajectories.}\label{fig:Btilde}
   \end{figure}

\paragraph{Acknowledgement} The work of F.L. is partially funded by NSF DMS-1913243. The work of H.L. is partially funded by NSF Award DMS-2108856.
The work of T.I. is partially funded by NSF Awards DMS-2012253 and CDS\&E-MSS-1953113. The authors would like to thank Prof. Charbel Farhat for helpful comments. 

\end{document}

%% file: notation.tex
\newtheorem{remark}{Remark}[section]

\newtheorem{theorem}{Theorem}[section]

\def\PP{{{\rm l}\kern - .15em {\rm P} }}
\def\PN2{{\PP_{N}-\PP_{N-2}}}

\newcommand{\bphi}{\boldsymbol{\varphi}}

\newcommand{\ba}{\boldsymbol{a}}

\newcommand{\bA}{\boldsymbol{A}}
\newcommand{\bb}{\boldsymbol{b}}

\newcommand{\bc}{\boldsymbol{c}}

\newcommand{\bff}{\boldsymbol{f}}

\newcommand{\bG}{{\boldsymbol G}}

\newcommand{\bu}{\boldsymbol{u}}

\newcommand{\bX}{\boldsymbol{X}}

\newcommand{\deleted}[1]{{}}
